\documentclass[a4, 12pt]{amsart}
\usepackage{fullpage}

\usepackage[title]{appendix}
\usepackage{graphicx}
\usepackage{subcaption}
\usepackage{amsthm}
\usepackage{mathrsfs} %%pour le style calligraphie en mathscr
\usepackage{amsfonts}
\usepackage{amssymb}
\usepackage{diagbox}
\usepackage{mathtools}
\usepackage{longtable} %tables with pagebreak
\usepackage{MnSymbol}
\usepackage{indentfirst}
\usepackage{listings,xcolor}
\usepackage[all]{xy}
\usepackage[colorlinks=true,linkcolor=blue]{hyperref}
\usepackage{ytableau}
\usepackage{tikz-cd}
\usepackage{enumerate}
\theoremstyle{plain}
\newtheorem{theorem}{Theorem}[section]

\newtheorem{conjecture}{Conjecture}

\newtheorem{lemma}[theorem]{Lemma}
\newtheorem{corollary}[theorem]{Corollary}
\newtheorem{proposition}[theorem]{Proposition}

\theoremstyle{remark}

\newtheorem{remark}{Remark}[section]

\makeatletter

\newcommand{\Rmnum}[1]{\expandafter\@slowromancap\romannumeral #1@}
\makeatother

\newcommand{\qbm}[2]{\left[ {#1 \atop #2} \right]}

\def\br{\mathbb R}
\def\bn{\mathbb N}
\def\bz{\mathbb Z}

\newmuskip\pFqmuskip

\numberwithin{equation}{section}
\allowdisplaybreaks[2] %allow page breaks in align enviroment

\begin{document}

\title[Positivity and tails of theta series]{Positivity and tails of pentagonal number series}
\author{Nian Hong Zhou}

\address{N. H. Zhou: School of Mathematics and Statistics, The Center for Applied Mathematics of Guangxi, Guangxi Normal University, Guilin 541004, Guangxi, PR China}
\email{nianhongzhou@outlook.com; nianhongzhou@gxnu.edu.cn}%

\thanks{This paper was partially supported by the National Natural Science Foundation of China (No. 12301423).}%

\begin{abstract}
In this paper, we refine a result of Andrews and Merca on truncated pentagonal number series. Subsequently, we establish some positivity results involving Andrews--Gordon--Bressoud identities and $d$-regular partitions. In particular, we prove several conjectures of Merca and Krattenthaler--Merca--Radu on truncated pentagonal number series.

\end{abstract}
\maketitle

\section{Introduction}
Throughout this paper,
let the \textit{$q$-shifted factorial} (cf.\  \cite{MR2128719})
be defined by
$$
(a;q)_\infty:=\prod_{j\ge 0}(1-aq^j),\quad\text{and}\quad
(a;q)_c:=\frac{(a;q)_\infty}{(aq^c;q)_\infty},
$$
for any indeterminant $a$ and complex number $c$. Products of $q$-shifted factorials are compactly denoted as
$$
(a_1,\ldots, a_m;q)_c:=\prod_{1\le j\le m}(a_j;q)_c
$$
for any integer $m\ge 1$. Further, for non-negative integers $n$ and $k$ the
\textit{$q$-binomial coefficient} is defined as
\begin{equation*}
\qbm{n}{k}_q:=
\frac{(q;q)_n}{(q;q)_k(q;q)_{n-k}}.
\end{equation*}

Euler's pentagonal number theorem
(cf.\ \cite[Equation~(8.10.10)]{MR2128719}) is stated as
\begin{equation}\label{PNT}
(q;q)_\infty=\sum_{j\ge 0}(-1)^jq^{j(3j+1)/2}(1-q^{2j+1}),
\end{equation}
which is one of the most famous q-series identities and plays an important role in the theory of partitions. In 2012, Andrews and Merca~\cite{MR2946378} gave an explicit expansion for
the averaged truncation of Euler's pentagonal number series appearing in \eqref{PNT}.
In particular, they \cite[Lemma~1.2]{MR2946378} proved that
\begin{align}\label{AM-id}
\frac{1}{(q;q)_\infty}\sum_{0\le j< k}(-1)^jq^{j(3j+1)/2}(1-q^{2j+1})=1+(-1)^{k-1}\sum_{n\ge k} \frac{q^{\binom{k}{2}+(k+1)n}}{(q;q)_n}
  \qbm{n-1}{k-1}_q.
\end{align}
As a direct consequence, they \cite[Theorem~1.1]{MR2946378} got that
\begin{align*}
(-1)^{k-1}\sum_{0\le j<k}(-1)^j
  \Big(p(n-j(3j+1)/2)-p(n-j(3j+5)/2-1) \Big)=M_k(n),
\end{align*}
where $p(n)$ represents the number of partitions of $n$, and $M_k(n)$ is the
number of partitions of $n$ in which $k$ is the least integer that
is not a part and there are more parts greater than $k$ than
there are less than $k$.  Equivalently, the above can be restated as
\begin{equation}\label{AMIE1}
\sum_{n\ge 1}M_k(n)q^n=\sum_{n\ge k} \frac{q^{\binom{k}{2}+(k+1)n}}{(q;q)_n}
  \qbm{n-1}{k-1}_q=\frac{(-1)^{k}}{(q;q)_\infty}\sum_{j\ge k}(-1)^jq^{j(3j+1)/2}(1-q^{2j+1}).
\end{equation}
Here we used Euler's pentagonal number theorem \eqref{PNT}. This gives the inequality, for $k\ge 1$,
\begin{equation}\label{AMIE0}
  (-1)^{k}\sum_{j\ge k}(-1)^j
  \Big(p(n-j(3j+1)/2)-p(n-j(3j+5)/2-1) \Big)\ge 0,
\end{equation}
with strict inequality if $n\ge k(3k+1)/2$.
\medskip

In this paper, we give a refinement of inequality \eqref{AMIE0}. For any distinct integers $\alpha,\beta,\gamma\ge 1$ such that $\gcd(\alpha,\beta)=\gcd(\beta,\gamma)=\gcd(\alpha,\gamma)=1$, and any integer $k\ge 1$, we define that
\begin{align}\label{eqm0}
G_{\alpha,\beta,\gamma}^{k}(q):=\sum_{n\ge 0}{\rm TP}_{\alpha,\beta,\gamma}^{k}(n)q^n=\frac{(-1)^k}{(1-q^\alpha)(1-q^\beta)(1-q^\gamma)}\sum_{j\ge k}(-1)^jq^{\frac{j(3j+1)}{2}}\left(1-q^{2j+1}\right).
\end{align}
Denoting that
\begin{align}\label{eq0000}
G_{\textsc{p}}^{k}(q)=\frac{(-1)^{k}}{(q;q)_\infty}\sum_{j\ge k}(-1)^jq^{j(3j+1)/2}(1-q^{2j+1}),
\end{align}
then it is clear that
$$G_{\textsc{p}}^{k}(q)=G_{\alpha,\beta,\gamma}^{k}(q)\cdot\prod_{\substack{j\ge 1\\ j\nin \{\alpha,\beta,\gamma\}}}\frac{1}{1-q^j}.$$
Thus, for the triple $(\alpha,\beta,\gamma)$ defined in equation \eqref{eqm0}, if ${\rm TP}_{\alpha,\beta,\gamma}^{k}(n)\ge 0$ for all integers $k\ge 1$ and $n\ge k(3k+1)/2$, with strict inequality when $n\in k(3k+1)/2+\{\alpha, \beta,\gamma\}$, this would provide a new proof for inequality \eqref{AMIE0}.

\medskip

It should be noted that a similar phenomenon, where it might be possible that fewer terms in the denominator of \eqref{eq0000} are needed and the series already possesses nonnegative coefficients, has been observed by Chan--Ho--Mao \cite[Section 5]{MR3531249} in their work on truncated series derived from the quintuple product identity. This observation was also made by Yao \cite{Yao2024} in her proof of the Ballantine and Merca conjecture \cite[Conjecture 2]{MR4632990} on the truncated sum of $6$-regular partitions.

\medskip

The main results of this paper are stated as follows.
\begin{theorem}\label{mth1}
Let $(\alpha, \beta,\gamma)\in\{(1,2,3),(1,2,5),(1,2,7), (1,3,4), (1,3,5)\}$.
For all integers $k\ge 1$ and $n\ge k(3k+1)/2$, we have ${\rm TP}_{\alpha,\beta,\gamma}^{k}(n)\ge 0$ with strict inequality except for the cases ${\rm TP}_{1,2,3}^{1}(13)=0$,
\begin{align*}
&{\rm TP}_{1,2,5}^{1}(n)=0,\;\text{for}~ n=11,13,15;\\
&{\rm TP}_{1,2,7}^{1}(n)=0,\;\text{for}~ n=7,9,11,13,14,15;\\
&{\rm TP}_{1,3,4}^{1}(n)=0,\;\text{for}~ n=11,13,14,17,38,41;\\
&{\rm TP}_{1,3,5}^{1}(n)=0,\;\text{for}~ n=10,11,13,14,16,37.
\end{align*}
In particular, ${\rm TP}_{\alpha,\beta,\gamma}^{k}(n)> 0$ for all integers $k> 1$ and $n\ge k(3k+1)/2$.
\end{theorem}
\begin{remark}
When the first version of this paper was submitted to arXiv, Ernest X. W. Xia informed us that the positivity of ${\rm TP}_{1,2,3}(n)$ had actually been proven by Yao \cite{Yao2024}.
\end{remark}

\begin{theorem}\label{mth2}Let $(\alpha, \beta,\gamma)\in\{(1,4,9),(2,3,5),(2,3,7)\}$. For all integers $k\ge 2$ and $n\ge k(3k+1)/2$, we have ${\rm TP}_{\alpha,\beta,\gamma}^{k}(n)\ge 0$ with strict inequality except for the cases
\begin{align*}
&{\rm TP}_{1,4,9}^{2}(n)=0,\;\text{for}~ n=21,24,25;\\
&{\rm TP}_{2,3,5}^{2}(n)=0,\;\text{for}~ n=20,23;\;\text{and}\;\;{\rm TP}_{2,3,5}^{k}\left(\frac{k(3k+1)}{2}+1\right)=0;\\
&{\rm TP}_{2,3,7}^{2}(3j)=0,\;\text{for}~ 4\le j\le 9, j\in\bz;\;\text{and}\;\;{\rm TP}_{2,3,7}^{k}\left(\frac{k(3k+1)}{2}+1\right)=0.
\end{align*}
\end{theorem}
\begin{remark}Our method can be used to determine the positivity of ${\rm TP}_{\alpha,\beta,\gamma}^{k}(n)$ with any triple $(\alpha,\beta,\gamma)$ given as in \eqref{eqm0}, for all integers $k\ge 1$ and $n\ge k(3k+1)/2$.
\end{remark}
The key idea in the proofs of Theorems \ref{mth1} and \ref{mth2} are based on the following classical result on partitions of an integer into a finite set of positive integers, see P\'{o}lya and Szeg\H{o} \cite[Problem 27.1, p.5, Part One]{MR1492447}.
\begin{proposition}\label{prom}
Define that
$$\sum_{n\ge 0}R_{\alpha,\beta,\gamma}(n)q^n=\frac{1}{(1-q^\alpha)(1-q^\beta)(1-q^\gamma)},$$
for any distinct integers $\alpha,\beta,\gamma\ge 1$ such that $\gcd(\alpha,\beta)=\gcd(\beta,\gamma)=\gcd(\alpha,\gamma)=1$.
Then for any $n\ge 0$,
$$R_{\alpha,\beta,\gamma}(n)=\frac{n^2+(\alpha+\beta+\gamma)n}{2\alpha\beta\gamma}+P_{\alpha,\beta,\gamma}(n),$$
where $P_{\alpha,\beta,\gamma}(n)$ is a periodic sequence of period $\alpha\beta\gamma$.
\end{proposition}
From equation \eqref{eqm0} and Proposition \ref{prom}, it is evident that ${\rm TP}_{\alpha,\beta,\gamma}^{k}(n)$ can be approximated as a finite alternating sum of quartic polynomials. As a result, the main term of ${\rm TP}_{\alpha,\beta,\gamma}^{k}(n)$ can be determined, and an estimate for its error can be obtained. Additionally, it can be observed that ${\rm TP}_{\alpha,\beta,\gamma}^{k}(n)$ exhibits polynomial growth. Therefore, by employing mathematical software like {\bf Mathematica}, we can successfully obtain complete proofs for Theorem \ref{mth1}. Further details are provided in Section \ref{sec3}.

\section{Proof of several conjectures on truncated pentagonal number series}

As applications of Theorems \ref{mth1} and \ref{mth2}, in this section, we prove several conjectures of Merca \cite{MR4344122, MR4316636, MR4309312} and Krattenthaler--Merca--Radu \cite{MR4378976} on truncated pentagonal number series. In particular, we generalize these conjectures to Andrews--Gordon--Bressoud identities and $d$-regular partitions, and provide proofs for them.

\subsection{Positivity and Andrews--Gordon--Bressoud identities}
In this subsection, we provide proofs and generalizations for two conjectures in Merca \cite{MR4309312} and Krattenthaler--Merca--Radu \cite{MR4378976}. Firstly, we focus on the truncations that involve the Andrews--Gordon--Bressoud identities (see \cite{MR351985} and \cite{MR556608}). It is well-known that
\begin{align}\label{AGI}
\sum_{n\ge 0}{\rm D}_{i,d}^{\tau}(n)q^n&:=\frac{(q^i,q^{2d+2+\tau-i},q^{2d+2+\tau};q^{2d+2+\tau})_\infty}{(q;q)_\infty}\nonumber\\
&=\sum_{r_1\ge r_2\ge \ldots\ge r_d\ge 0}\frac{q^{r_1^2+\cdots+r_d^2+r_i+\cdots+r_d}}{(q;q)_{r_1-r_2}(q;q)_{r_2-r_3}\cdots(q;q)_{r_{d-1}-r_d}(q^{2-\tau};q^{2-\tau})_{r_d}},
\end{align}
where $\tau\in\{0,1\}$, $d, i$ are integers such that $d\ge 1$ and $1\le i\le d+1$. Notice that ${\rm D}_{i,d}^\tau(n)$ is the number of partitions of $n$ into parts not congruent to $0$, $i$, or $-i$ $\pmod {2d+2+\tau}$.

\medskip

For any integers $d\ge 4, 1\le i\le d/2$ and $k\ge 1$, we define that
\begin{align}\label{eqdid}
\mathcal{D}_{i,d}^k(q):=\sum_{n\ge 0}C_{i,d}^k(n)q^n=(-1)^{k-1}\left(\frac{1}{(q;q)_\infty}\sum_{-k< n\le k}(-1)^nq^{n(3n-1)/2}-1\right)(q^i,q^{d-i},q^{d};q^{d})_\infty.
\end{align}
Using Jacobi triple product identity (cf.\ \cite[Appendix~(II.28)]{MR2128719}), which is
$$(q^i,q^{d-i},q^{d};q^{d})_\infty=\sum_{n\in\bz}(-1)^nq^{d\binom{n}{2}+in},$$
it is clear that
\begin{equation*}
(-1)^{k-1}\sum_{-k<\ell\le k}(-1)^\ell {\rm D}_{i,d}^{\tau}(n-\ell(3\ell-1)/2)+(-1)^k {\bf 1}_{n\in \mathscr{P}_{2d+2+\tau,i}}=C_{i,2d+2+\tau}^k(n).
\end{equation*}
Here and throughout this section, let $\mathscr{P}_{a,b}=\left\{a\binom{n}{2}+bn: n\in\bz\right\}$. Moreover, let ${\bf 1}_{event}$ denote the indicator function.

\medskip

Now we are ready to state the main result of this subsection.

\begin{theorem}\label{th2}Let ${\rm D}_{i,d}^{\tau}(n)$ be defined as in \eqref{AGI}. Then, for any $2-\tau\le i \le d+\tau$, all integers $k\ge 1$, we have
$$C_{i,2d+2+\tau}^k(n)=(-1)^{k-1}\sum_{-k<\ell\le k}(-1)^\ell {\rm D}_{i,d}^{\tau}(n-\ell(3\ell-1)/2)+(-1)^k {\bf 1}_{n\in \mathscr{P}_{2d+2+\tau,i}},$$
are all zero for $0\le n<k(3k+1)/2$. For $n\ge  k(3k+1)/2$ they are positive except for the cases
\begin{align*}
C_{2,5}^1(5)&=C_{2,5}^1(7)=C_{2,5}^1(9)=C_{2,5}^1(11)=0,\\
C_{1,5}^1(7)&=C_{1,5}^1(13)=-1,\\
C_{1,5}^1(9)&=C_{1,5}^1(11)=C_{1,5}^2(12)=0,\\
C_{1, 7}^1(9)&=C_{1, 2d+3}^1(5)=C_{1, 2d+3}^1(7)=C_{1,2d+3}^k\left(k(3k+1)/2+1\right)=0.
\end{align*}
\end{theorem}
Before giving the proof of this theorem, we first provide two immediate corollaries. By using Andrews--Merca's identity \eqref{AM-id}, we can express the $q$-series $\mathcal{D}_{i,d}^k(q)$ in the following form:
\begin{align*}
(q^i,q^{d-i},q^{d};q^{d})_\infty&\sum_{n\ge k} \frac{q^{\binom{k}{2}+(k+1)n}}{(q;q)_n}
  \qbm{n-1}{k-1}_q\\
  =(-1)^{k-1}&\left(\frac{1}{(q;q)_\infty}\sum_{-k< n\le k}(-1)^nq^{n(3n-1)/2}-1\right)(q^i,q^{d-i},q^{d};q^{d})_\infty.
\end{align*}
Therefore, the case $(d,\tau)=(1,1)$ in Theorem \ref{th2} yields the following corollary for the Rogers--Ramanujan identities, which were conjectured by Merca \cite[Conjectures 4.1, 4.2]{MR4309312}.
\begin{corollary}For $k>0$, the expression
$$(q^2,q^3, q^5;q^5)_\infty\sum_{n\ge k} \frac{q^{\binom{k}{2}+(k+1)n}}{(q;q)_n}
  \qbm{n-1}{k-1}_q$$
has non-negative coefficients.
For $k>1$, the expression
\begin{equation}\label{eqrr1}
(q,q^4, q^5;q^5)_\infty\sum_{n\ge k} \frac{q^{\binom{k}{2}+(k+1)n}}{(q;q)_n}
  \qbm{n-1}{k-1}_q
\end{equation}
has non-negative coefficients. For $k=1$, the coefficients of $q^7$ and $q^{13}$ in \eqref{eqrr1} are equal to $-1$ while any other coefficient is non-negative.
\end{corollary}
Moreover, the case $d=2$ with $i=1$ in Theorem \ref{th2} yields the following corollary for an Andrews--Gordon identity modulo $7$, which was conjectured by Krattenthaler--Merca--Radu \cite[Conjecture 40]{MR4378976}.\footnote{In \cite{MR4378976}, Conjecture 40 states that it holds for all $k>0$ rather than $k>1$. Through {\bf Mathematica} verification, it is observed that for $k=1$, the coefficients also equal $0$ for $n\in \{5,7,9\}$. Therefore, the conjecture should only hold for $k>1$.}
\begin{corollary}For $k > 1$, the coefficients of $q^n$ in the series
$$\frac{(-1)^k}{(q^2,q^3,q^4,q^5;q^7)_\infty}\sum_{n\ge k}(-1)^nq^{n(3n+1)/2}\left(1-q^{2n+1}\right)$$
are all zero for $0\le n<k(3k+1)/2$ and $n=k(3k+1)/2+1$. For $n=k(3k+1)/2$ and $n\ge k(3k+1)/2+2$ all the coefficients are positive.
\end{corollary}
We now give the proof of Theorem \ref{th2}.
\begin{proof}[Proof of Theorem \ref{th2}]
Using Euler's pentagonal number theorem \eqref{PNT} to \eqref{eqdid} implies
\begin{align}\label{eqdid1}
\mathcal{D}_{i,d}^k(q)=\sum_{n\ge 0}C_{i,d}^k(n)q^n=(-1)^{k}\sum_{n\ge k}(-1)^nq^{n(3n+1)/2}\left(1-q^{2n+1}\right)\prod_{\substack{n\ge 1\\ n\not\equiv 0, \pm i\pmod d}}\frac{1}{1-q^n}.
\end{align}
Therefore, the coefficients $C_{i,d}^k(n)$ are all zero for $0\le n<k(3k+1)/2$. We only need to consider the case $n\ge k(3k+1)/2$.
\medskip
Note that
\begin{equation}\label{eq100}
\mathcal{D}_{i,d}^k(q)=G_{\alpha,\beta,\gamma}^k(q)\cdot \prod_{\substack{n\ge 1, n\not\in\{\alpha,\beta,\gamma\}\\ n\not\equiv 0,\pm i\pmod d}}\frac{1}{1-q^n},
\end{equation}
where $G_{\alpha,\beta,\gamma}^k(q)$ is defined by \eqref{eqm0}, and the triple $(\alpha, \beta,\gamma)$ can be chosen for each pair $(i,d)$ as shown in Table \ref{tb0}.
\begin{table}[ht]
\resizebox{0.65\textwidth}{!}{
\begin{tabular}{|c|c|c|c|c|c|}
  \hline
 \diagbox{\quad $i$}{$(\alpha,\beta,\gamma)$}{$d$}& $5$ & $6$ & $7$ & $8$ & $\ge 9$ \\
     \hline
  1 & $(2,3,7)$ & -- & $(2,3,5)$ & $(2,3,5)$ & (2,3,5) \\
    \hline
  2 & $(1,4,9)$ & $(1,3,5)$ & $(1,3,4)$ & $(1,3,4)$ & $(1,3,4)$ \\
    \hline
  3 & -- & -- & $(1,2,5)$ & $(1,2,7)$ & $(1,2,5)$  \\
    \hline
  $\ge 4$ & -- & -- & -- & -- & $(1,2,3)$ \\
  \hline
\end{tabular}
}
\captionof{table}{The triples $(\alpha,\beta, \gamma)$ for $\mathcal{D}_{i,d}^k(q)$}\label{tb0}
\end{table}

\medskip

Combining with Table \ref{tb0}, it can be observed that the positivity of $C_{i,d}^k(n)$ for all $k\ge 1$ and for the cases $d\ge 6$ and $2\le i<d/2$ will follow from Theorem \ref{mth1}. While by Theorem \ref{mth2}, we can only establish the positivity of $C_{2,5}^k(n), C_{1,5}^k(n)$, and $C_{1,d}^k(n)$ for $d\ge 7$ and all $k\ge 2$. For the case where $k=1$, we need to use different approaches. Specifically, our proof of Theorem \ref{th2} will follow from the following Lemmas \ref{lem11}--\ref{lem13}.
\end{proof}
\begin{lemma}\label{lem11}
Let $k,i, d,n$ be positive integers such that $2\le i< d/2$ and $d\ge 6$. Then, the coefficients $C_{i,d}^k(n)$ are positive for all $n\ge k(3k+1)/2$.
\end{lemma}
\begin{proof}From equation \eqref{eq100} and Table \ref{tb0}, the use of Theorem \ref{mth1} implies that for all integers $k\ge 2$, $d\ge 6$ and $2\le i<d/2$, the coefficients $C_{i,d}^k(n)$ are positive for any $n\geq {k(3k+1)}/{2}$. Moreover, for $k=1$, the coefficients  $C_{i,d}^k(n)$ are positive for all $n\ge 42$. Using {\bf Mathematica}, one can verify that for all integers $6\le d\le 13$ and $2\le i<d/2$, the coefficients $C_{i,d}^k(n)$ are positive for all $2\le n\le 42$. It remains to prove the lemma for the cases $d\ge 14$ with $k=1$.

\medskip

For the cases $d\ge 14$ with $i=2$, the use of Theorem \ref{mth1} for $G_{1,3,4}^{1}(q)$ implies that $C_{2,d}^1(n)$ are positive for all $n\ge 2$, except for the possible case $n\in S_{134}:=\{11,13,14,17,38,41\}$. Since for $d\ge 14$, the factor $1/(1-q^5)(1-q^{6})(1-q^{8})$ appears in the infinite product
$$\prod\limits_{\substack{n\ge 1, n\neq 1,3,4\\ n\not\equiv 0, \pm 2\pmod d}}\frac{1}{1-q^n},$$
and one can check that each $n\in S_{134}$ can be write as a sum of $5,6$ and $8$. Thus all the coefficients $C_{2,d}^1(n)$ with $n\in S_{134}$ are positive.

\medskip

Using the similar arguments to the above, one can prove the positivity of $C_{i,d}^1(n)$ for the cases $d\ge 14$ with $i=3$.

\medskip

For the cases $d\ge 14$ with $i\ge 4$, using Theorem \ref{mth1} for $G_{1,2,3}^{1}(q)$, we can see that $C_{i,d}^1(n)$ are positive for all $n\ge 2$ except for the possible case $n=13$. Since the factors $1/(1-q^{i+1})(1-q^{12-i})$ and $1/(1-q^{5})(1-q^{8})$, appears in the infinite product
$$\prod\limits_{\substack{n\ge 1, n\neq 1,2,3\\ n\not\equiv 0, \pm i\pmod d}}\frac{1}{1-q^n},$$
for $4\le i\le 8$ and $i\ge 9$, respectively. Therefore, the coefficient of $q^{13}$ in the above product is positive. This completes the proof.
\end{proof}
\begin{lemma}\label{lem12}
Let $k\ge 2, d$ and $n$ be positive integers. For all $n\ge k(3k+1)/2$, the coefficients $C_{2,5}^k(n)$ are positive.
For all $d\ge 5, d\neq 6$ and for all $n\ge k(3k+1)/2$, the coefficients $C_{1,d}^k(n)$ are positive except for the cases $C_{1,5}^2(12)=C_{1,d}^k\left(k(3k+1)/2+1\right)=0$.
\end{lemma}
\begin{proof}
We first prove the lemma for the case $d=5$. From equation \eqref{eq100} and Table \ref{tb0}, the use of Theorem \ref{mth2} for $G_{1,4,9}^k(q)$ and $G_{2,3,7}^k(q)$ implies that for $k\ge 3$, the coefficients $C_{2,5}^k(n)$ are positive for all $n\geq {k(3k+1)}/{2}$, and the coefficients $C_{1,5}^k(n)$ are positive for all $n\geq {k(3k+1)}/{2}, n\neq k(3k+1)/2+1$. Furthermore, for $k=2$, the coefficients $C_{1,5}^2(n)$ and $C_{2,5}^2(n)$ are positive for all $n\ge 28$. Using {\bf Mathematica}, we can verify the remaining cases $7\le n\le 27$. This completes the proof of the lemma for $d=5$.

\medskip

We now prove the lemma for the cases $d\ge 7$. From equation \eqref{eq100} and Table \ref{tb0}, the use of Theorem \ref{mth2} for $G_{2,3,5}^k(q)$ implies that for $k\ge 3$, the coefficients $C_{1,d}^k(n)$ are positive for all $n\geq {k(3k+1)}/{2}, n\neq k(3k+1)/2+1$.  Furthermore, for $k=2$, using Theorem \ref{mth2} for $G_{2,3,5}^{2}(q)$, we can see that $C_{1,d}^2(n)$ are positive for all $n\ge 7$ except for the possible case $n\in \{8, 20, 23\}$. Since the factors $1/(1-q^4)(1-q^{11})$ and $1/(1-q^4)(1-q^{7})$, appears in the infinite product
$$\prod\limits_{\substack{n\ge 1, n\neq 2,3,5\\ n\not\equiv 0, \pm 1\pmod d}}\frac{1}{1-q^n},$$
for $d\in \{7,8\}$ and $d\ge 9$, respectively. Notice that $20=5\cdot 4$ and $23=3\cdot 4+11=4\cdot 4+7$, we have
the coefficients of $q^{20}, q^{23}$ in the series for the above product are positive. It remains to prove $C_{1,d}^2(8)=0$. In fact, by equation \eqref{eqdid1} we have
\begin{align*}
\sum_{n\ge 0}C_{1,d}^2(n)q^n&=\sum_{n\ge 0}(-1)^nq^{n(3n+1)/2+6n+7}\left(1-q^{2n+5}\right)\prod_{\substack{n\ge 1\\ n\not\equiv 0, \pm 1\pmod d}}\frac{1}{1-q^n}.
\end{align*}
Notice that the factor $1/(1-q)$ does not appear in the above infinite product, which immediately implies $C_{1,d}^2(8)=0$. Therefore, for $k\ge 2$ the coefficients $C_{1,d}^k(n)$ are positive for all $n\geq {k(3k+1)}/{2}, n\neq k(3k+1)/2+1$. This completes the proof.
\end{proof}
\begin{lemma}\label{lem13}
For $i\in\{1,2\}$ and for all integers $n\ge 2$, the coefficients $C_{i, 5}^1(n)$ are positive
except for the cases
\begin{align*}
C_{1,5}^1(7)&=C_{1,5}^1(13)=-1,\\
C_{1,5}^1(3)&=C_{1,5}^1(5)=C_{1,5}^1(9)=C_{1,5}^1(11)=0,\\
C_{2,5}^1(5)&=C_{2,5}^1(7)=C_{2,5}^1(9)=C_{2,5}^1(11)=0.
\end{align*}
For all integers $d$ and $n\ge 2$, the coefficients $C_{1, 2d+3}^1(n)$ are positive
except for the cases
$$C_{1, 2d+3}^1(3)=C_{1, 2d+3}^1(5)=C_{1, 2d+3}^1(7)=C_{1, 7}^1(9)=0.$$
\end{lemma}
\begin{proof}
Notice that the generating function \eqref{eqdid} of $C_{i,d}^1(n)$ can reduce to
$$\sum_{n\ge 0}C_{i,d}^1(n)q^n=(1-q)\frac{(q^i,q^{2d+3-i},q^{2d+3};q^{2d+3})_\infty}{(q;q)_\infty}-\sum_{n\in\bz}(-1)^n q^{(2d+3)\binom{n}{2}+in}.$$
Using Andrews--Gordon identity \eqref{AGI}, for any integers $d\ge 1$, with $r_{d+1}=0$, we have
\begin{align*}
&(1-q)\frac{(q^i,q^{2d+3-i},q^{2d+3};q^{2d+3})_\infty}{(q;q)_\infty}\\
=&(1-q)\sum_{r_1\ge r_2\ge \ldots\ge r_d\ge 0}\frac{q^{r_1^2+\cdots+r_d^2+r_i+\cdots+r_d}}{(q;q)_{r_1-r_2}\cdots(q;q)_{r_{d-1}-r_d}(q;q)_{r_d}}\nonumber\\
=&1-q+\sum_{1\le h\le d}\sum_{\substack{r_h>r_{h+1}\\ r_1\ge r_2\ge \ldots\ge r_d\ge 0 }}\frac{1-q}{(q;q)_{r_{h}-r_{h+1}}}\frac{(q;q)_{r_{h}-r_{h+1}}q^{r_1^2+\cdots+r_d^2+r_i+\cdots+r_d}}{(q;q)_{r_1-r_2}\cdots(q;q)_{r_{d-1}-r_d}(q;q)_{r_d}}\nonumber\\
=& 1-q+\sum_{r_1> r_2=\ldots=r_d=0}\frac{(1-q)q^{r_1^2+r_1\cdot {\bf 1}_{i=1}}}{(q;q)_{r_1}}+E_{i,d}(q),
\end{align*}
where $E_{i,d}(q)$ is a power series in $q$ with non-negative coefficients. By an easy simplification, we further have
\begin{align*}
&(1-q)\frac{(q^i,q^{2d+3-i},q^{2d+3};q^{2d+3})_\infty}{(q;q)_\infty}\\
=&1-q+\sum_{r\ge 1}\frac{q^{r^2+r\cdot {\bf 1}_{i=1}}}{(q^2;q)_{r-1}}+E_{i,d}(q)\\
=&1-q+q^{1+{\bf 1}_{i=1}}+\frac{q^{2^2+2\cdot {\bf 1}_{i=1}}}{1-q^2}+\frac{q^{3^2+3\cdot {\bf 1}_{i=1}}}{(1-q^2)(1-q^3)}+\sum_{r\ge 3}\frac{q^{r^2+r\cdot {\bf 1}_{i=1}}}{(q^2;q)_{r-1}}+E_{i,d}(q).
\end{align*}
Therefore,
$$\left(\frac{1-q}{(q;q)_\infty}-1\right)(q^i,q^{2d+3-i},q^{2d+3};q^{2d+3})_\infty=\frac{q^{9+3\cdot {\bf 1}_{i=1}}}{(1-q^2)(1-q^3)}-q-\sum_{n\in\bz}(-1)^n q^{(2d+3)\binom{n}{2}+in}+E_{i,d}^1(q),$$
where $E_{i,d}^1(q)$ is a power series in $q$ with non-negative coefficients.  Note that
$$\frac{1}{(1-q^2)(1-q^3)}=\sum_{n\ge 0}\left(\lceil (n+1)/6\rceil-{\bf 1}_{6\mid(n-1)}\right)q^n,$$
where $\lceil x\rceil$ denote the smallest integer $\ge x$ for any real $x$. This implies that in the power series of $1/(1-q^2)(1-q^3)$, the coefficients of $q^n$ are great than $2$ for all integers $n\ge 8$. Thus the coefficients $C_{i,2d+3}^1(n)$ are positive for all $n\ge 20$.  For $2\le n<20$, using {\bf Mathematica}, one can verify the cases $d=1$ with $i\in\{1,2\}$, and $2\le d\le 7$ with $i=1$. It remains to prove the cases for $d\ge 8$ with $i=1$.

\medskip

For $d\ge 8$, note that
\begin{align*}
\left(\frac{1-q}{(q;q)_\infty}-1\right)(q,q^{2d+2},q^{2d+3};q^{2d+3})_\infty&=\left(1-q+O(q^{2d+2})\right)\left(\frac{1}{(q^2;q)_\infty}-1\right)\\
&=\frac{1-q}{(q^2;q)_\infty}-1+q+O(q^{20})\\
&=q^2 + q^4 + 2 q^6 +\sum_{j\ge 8}c_jq^j+O(q^{20}),
\end{align*}
for some $c_j>0, (j\ge 8)$. That is, the coefficients $C_{1,2d+3}^1(n)$ are positive for all $d\ge 8$ and $2\le n<20$, except for the case $n\in\{3,5,7\}$. This completes the proof.
\end{proof}

\subsection{Positivity and $d$-regular partition functions}
In this subsection, we prove and generalize two conjectures in Merca \cite{MR4344122, MR4316636}, and Ballantine and Merca \cite{MR4632990}.
In particular, we prove the following Theorem \ref{th1} involving the $d$-regular partition function $b_{d}(n)$, defined as:
\begin{align*}
\sum_{n\ge 0}b_{d}(n)q^n=\frac{(q^d;q^d)_\infty}{(q;q)_\infty}=\frac{1}{(q,q^2,\ldots, q^{d-1};q^d)_\infty},
\end{align*}
for any integer $d\ge 2$.

\begin{theorem}\label{th1}For any integer $k\ge 1$ and $d\ge 2$,  the coefficients of $q^n$ in the series
$$(-1)^{k-1}\left(\frac{1}{(q;q)_\infty}\sum_{-k<n\le k}(-1)^nq^{n(3n-1)/2}-1\right)(q^d;q^d)_\infty$$
are all zero for $0\le n<k(3k+1)/2$, and for $n\ge k(3k+1)/2$ all the coefficients are positive.
\end{theorem}
\begin{remark}We have the following remarks for Theorem \ref{th1}.
\begin{enumerate}
  \item The case $d=2$ was conjectured by Merca in \cite[Conjecture 4.1]{MR4316636}.
  \item The case $d=6$ was conjectured in the work of Ballantine and Merca \cite[Conjecture 2]{MR4632990}. When the first version of this paper was submitted to arXiv, Ernest X. W. Xia informed us that Yao \cite{Yao2024} had already proved this conjecture. Moreover, Ding and Sun \cite{DingSun2024} proved that for $d\ge 3$, the coefficients are non-negative for all $n\ge 0$.
\end{enumerate}
\end{remark}
From Theorem \ref{th1}, we can obtain the following result involving partitions into parts not congruent to $0, \pm 3 \pmod {12}$, which was conjectured by Merca in \cite{MR4344122}.
\begin{corollary}[{Merca \cite[Conjecture 4.1]{MR4344122}}]For any integer $k\ge 1$,  the coefficients of $q^n$ in the series
$$(-1)^{k-1}\left(\frac{1}{(q;q)_\infty}\sum_{-k<n\le k}(-1)^nq^{n(3n-1)/2}-1\right)(q^3,q^9;q^{12})_\infty$$
are all zero for $0\le n<k(3k+1)/2$, and are all positive for $n\ge k(3k+1)/2$.
\end{corollary}

\begin{proof}Just by noting that $$(q^3,q^9;q^{12})\infty=\frac{(q^3;q^3)\infty}{(q^6,q^{12};q^{12})_\infty},$$
then, the use of Theorem \ref{th1} for $d=3$ implies the proof.
\end{proof}
From Theorem \ref{th1}, we can readily derive the following infinite family of inequalities for $d$-regular partition functions.
\begin{corollary}\label{cor11}For all integers $d\ge 2$, $k\ge 1$ and $n\ge  0$,
$$(-1)^{k-1}\sum_{-k<\ell\le k}(-1)^\ell b_d(n-\ell(3\ell-1)/2)+(-1)^k {\bf 1}_{n/d\in \mathscr{P}_{3, 1}}\ge 0,$$
with strict inequality if $n\ge k(3k+1)/2$.
\end{corollary}

We now give the proof of Theorem \ref{th1}.
\begin{proof}[Proof of Theorem \ref{th1}]Using Euler's pentagonal number theorem \eqref{PNT}, we have
\begin{align*}
B_d^k(q):=&(-1)^{k-1}\left(\frac{(q^d;q^d)_\infty}{(q;q)_\infty}\sum_{-k<n\le k}(-1)^nq^{n(3n-1)/2}-(q^d;q^d)_\infty\right)\\
=&(-1)^k(q^d;q^d)_\infty-(-1)^k(q^d;q^d)_\infty+(-1)^k\frac{(q^d;q^d)_\infty}{(q;q)_\infty}\sum_{n\ge k}(-1)^nq^{n(3n+1)/2}\left(1-q^{2n+1}\right)\\
=&\frac{(-1)^{k}}{(q,q^2,\ldots, q^{d-1};q^d)_\infty}\sum_{n\ge k}(-1)^nq^{n(3n+1)/2}\left(1-q^{2n+1}\right).
\end{align*}
Therefore, the coefficients of $q^n$ in the above series are all zero for $0\le n<k(3k+1)/2$. We only need to consider the cases $n\ge k(3k+1)/2$. For $d=2,3$, we note that
$$B_2^k(q)=G_{1,3,5}^{k}(q)\cdot\frac{1}{(q^7;q^2)_\infty}\;\;
\text{and}\;\; B_3^k(q)=G_{1,2,5}^{k}(q)\cdot\frac{1}{(q^4,q^8;q^3)_\infty}.$$
The use of Theorem \ref{mth1} implies that for all $n\ge {k(3k+1)}/{2}$, the coefficients of $q^n$ in the series for $B_2^k(q)$ are positive, except for the possible case $(k,n)\in \{1\}\times\{10,11,13,14,16,37\}$. Similarly, the coefficients of $q^n$ in the series for $B_3^k(q)$ are positive, except for the possible case $(k,n)\in \{1\}\times\{11,13,15\}$. These exceptional cases can be easily verified using {\bf Mathematica}. This completes the proof for $d=2,3$.

\medskip

For any $d\ge 4$, we note that
$$B_d^k(q)=G_{1,2,3}^{k}(q)\cdot\frac{1}{(q^{d+1},q^{d+2},q^{d+3};q^d)_\infty}\prod_{4\le j<d}\frac{1}{(q^{j};q^d)_\infty}.$$
The use of Theorem \ref{mth1} implies that that for all $n\ge {k(3k+1)}/{2}$, the coefficients of $q^n$ in the series for $B_d^k(q)$ are positive, except for the possible case $(k, n)=(1,13)$. This case for $d=13$ can be easily verified using {\bf Mathematica}, while for the cases $d\neq 13$, since the factor $(1-q^{13})$ appears in the infinite product $(q^{d+1},q^{d+2},q^{d+3};q^d)_\infty \prod_{4\le j<d}(q^{j};q^d)_\infty$, which leads to the positivity of the coefficient of $q^{13}$. This completes the proof.
\end{proof}
%In view of
%$$(q^3,q^9;q^{12})_\infty=\frac{(q^3;q^3)_\infty}{(q^6,q^{12};q^{12})_\infty},$$
%from Theorem \ref{th1}, we immediately obtain the following corollary, which was conjectured by Merca in \cite{{MR4344122}}.
%\begin{corollary}[{Merca \cite[Conjecture 4.1]{MR4344122}}]For any integer $k\ge 1$,  the coefficients of $q^n$ in the series
%$$(-1)^{k-1}\left(\frac{1}{(q;q)_\infty}\sum_{-k<n\le k}(-1)^nq^{n(3n-1)/2}-1\right)(q^3,q^9;q^{12})_\infty$$
%are all zero for $0\le n<k(3k+1)/2$, and are all positive for $n\ge k(3k+1)/2$.
%\end{corollary}
%Using \eqref{AMIE}, we can state Theorem \ref{th1} as the following form:
%\begin{equation}
%(q^{d};q^{d})_\infty\sum_{n\ge k} \frac{q^{\binom{k}{2}+(k+1)n}}{(q;q)_k}\qbm{n-1}{k-1}_q
%\end{equation}

\section{Proof of Theorems \ref{mth1} and \ref{mth2}}\label{sec3}

In this section, we give the proofs of Theorems \ref{mth1} and \ref{mth2}. Before giving the proofs, we first prove a more general result, which is the following Theorem \ref{main}.

\subsection{The fundamental result}
Throughout this section, let $a>0, b\ge 0$ and $c\in\br$. Define $f(n)=0$ for all $n<0$ and
\begin{equation}\label{eqf}
f(n)=an^2+bn+c+B(n),
\end{equation}
for $n\ge 0$, where $B(n)$ is a bounded sequence of real numbers bounded by $B_f:=\sup_n|B(n)|$. We consider the sequence $F_k(n)$, which is defined by
\begin{align}\label{eqF0}
\sum_{n\ge 0}F_k(n)q^n=\sum_{n\ge 0}f(n)q^n\sum_{j\ge k}(-1)^{j-k}q^{j(3j+1)/2}\left(1-q^{2j+1}\right).
\end{align}

The main result of this section is the following Theorem \ref{main}, which is also
indeed the heart of this paper. We denote that
\begin{equation}\label{eqabkf}
A_F=\frac{B_f+4a/3+b/2}{a},\;\; B_F=\frac{B_f+4 a/3+|c|}{a},\;\; k_F=\frac{1}{3}\left(1+\sqrt{1+3(2A_F+B_F)}\right)
\end{equation}
and for each integer $k\ge 1$, we define that
\begin{equation}\label{eqlfk}
\ell_F(k)=\left({1}/{6}+\frac{A_F}{3k-2}-k\right)+\sqrt{\left({1}/{6}+\frac{A_F}{3k-2}-k\right)^2+\left(k+{2}/{3}+\frac{B_F}{3k-2}\right)}.
\end{equation}
The main result of this section is stated as the following theorem.
\begin{theorem}\label{main}Let $k$ and $n$ be positive integers.
\begin{enumerate}
  \item  For all $n<k(3k+1)/2+2k+1$, we have $F_k(n)=f(n-k(3k+1)/2)$.
   \item  For all $n\ge k(3k+1)/2+2k+1$ with $k>k_F$, we have $F_k\left(n\right)>0$.
  \item  For all $n\ge k(3k+1)/2+(3\ell-1)(2k+\ell)/2$ with any $\ell> \ell_F(k)$, we have $F_k\left(n\right)>0$.
\end{enumerate}
\end{theorem}
\begin{remark}Under the above theorem, to determine the positivity of $F_k(n)$, it remains to verify the finite case of $1\le k\le k_F$ with
$$k(3k+1)/2+2k+1\le n<k(3k+1)/2+(3\ell_F(k)-1)(2k+\ell_F(k))/2.$$
For a given function $f$ defined by the above, it can be determined using scientific calculation tools such as {\bf Mathematica} and {\bf Mape}.
\end{remark}
We now give the proof of Theorem \ref{main}. From \eqref{eqF0}, we see that
\begin{equation}\label{eqF}
F_k(n)=\sum_{j\ge k}(-1)^{j-k}\left(f(n-j(3j+1)/2)-f(n-1-j(3j+5)/2)\right).
\end{equation}
Hence if $n<1+k(3k+5)/2=k(3k+1)/2+2k+1$, then $F_k(n)=f(n-k(3k+1)/2)$. In the following, we only consider the case for
$$n\ge k(3k+1)/2+2k+1.$$
Let $k, m\in\bz$ such that $m>k\ge 1$. Note that each integer $n\ge k(3k+1)/2+2k+1$ can be uniquely write as
$$n=\frac{1}{2}m(3m-1)+h,$$
for some integers $m,h$ such that $0\le h\le 3m$. We need the following proposition.
\begin{proposition}\label{lem21}Let $k, m, h\in\bz$ such that $m>k\ge 1$ and $0\le h\le 3m$. Let $f(\cdot)$ and $F_k(\cdot)$ be defined as in \eqref{eqf} and \eqref{eqF}, respectively. Then, we have
\begin{align*}
F_k\left(\frac{m(3m-1)}{2}+h\right)=&(2a+b)k-3a k^3+ak \big(3m^2+(2h-m)\big)\\
&-(-1)^{m-k}\big( (2a+b)m+am (2h-m)-f(h-m)\big)+E_k(m,h).
\end{align*}
where $|E_k(m,h)|\le 2(m-k)B_f$.
\end{proposition}
\begin{proof}Notice that $0\le h\le 3m$.  If $j\ge m+1$ then we have
\begin{align*}
\left(\frac{m(3m-1)}{2}+h\right)-\frac{j(3j+1)}{2}&\le \frac{m(3m-1)}{2}+h-\frac{(m+1)(3m+4)}{2}\nonumber\\
&=h-2-4m<0;
\end{align*}
if $j\ge m$ then we have
\begin{align*}
\left(\frac{m(3m-1)}{2}+h\right)-1-\frac{j(3j+5)}{2}&\le \frac{m(3m-1)}{2}+h-1-\frac{m(3m+5)}{2}\nonumber\\
&=h-1-3m<0;
\end{align*}
if $0\le j\le m-1$ then we have
\begin{align*}
\left(\frac{m(3m-1)}{2}+h\right)-1-\frac{j(3j+5)}{2}&\ge \frac{m(3m-1)}{2}+h-1-\frac{(m-1)(3m+2)}{2}\nonumber\\
&=h\ge 0.
\end{align*}
Therefore, with $n={m(3m-1)}/{2}+h$, by noting that $f(j)=0$ for $j<0$ and using \eqref{eqF} implies
\begin{align*}
F_k\left(n\right)=&\sum_{k\le j<m}(-1)^{j-k}\left(f\left(n-\frac{j(3j+1)}{2}\right)-f\left(n-1-\frac{j(3j+5)}{2}\right)\right)\\
&+(-1)^{m-k}f\left(\frac{m(3m-1)}{2}+h-\frac{m(3m+1)}{2}\right).
\end{align*}
Denoting by $\hat{f}(x)=ax^2+bx$, then by \eqref{eqf} and the above we have
\begin{align*}
F_k(n)=&\sum_{k\le j< m}(-1)^{j-k}\left(\hat{f}\left(n-\frac{j(3j+1)}{2}\right)-\hat{f}\left(n-1-\frac{j(3j+5)}{2}\right)\right)\\
&+\sum_{k\le j< m}(-1)^{j-k}\left(B\left(n-\frac{j(3j+1)}{2}\right)-B\left(n-1-\frac{j(3j+5)}{2}\right)\right)+(-1)^{m-k}f(h-m)\\
=&\sum_{k\le j< m}(-1)^{j-k}\left(\hat{f}\left(n-\frac{j(3j+1)}{2}\right)-\hat{f}\left(n-1-\frac{j(3j+5)}{2}\right)\right)\\
&+(-1)^{m-k}f(h-m)+E_k(m,h),
\end{align*}
where $|E_k(m,h)|\le 2(m-k)B_f$. By using induction, one can verify that
\begin{align*}
&\sum_{k\le j< m}(-1)^{j-k}\left(\hat{f}\left(n-\frac{j(3j+1)}{2}\right)-\hat{f}\left(n-1-\frac{j(3j+5)}{2}\right)\right)\\
=&k(b+a(2(n+1)-3k^2))-(-1)^{m-k}m(b+a(2(n+1)-3m^2)).
\end{align*}
Therefore, by substituting $n=m(3m-1)/2+h$ into the above, we have
\begin{align*}
F_k(n)=&(2a+b)k-3ak^3+ak \big(3m^2+(2h-m)\big)\\
&-(-1)^{m-k}\big( (2a+b)m+am (2h-m)-f(h-m)\big)+E_k(m,h).
\end{align*}
This completes the proof of the proposition.
\end{proof}

\begin{proof}[Completing the proof of Theorem \ref{main}]
Using Proposition \ref{lem21}, we have the following statements. For $0\le h<m$, we have $f(h-m)=0$ and
\begin{align*}
F_k(n)=&(2a+b)k-3ak^3+ak \big(3m^2+(2h-m)\big)\\
&-(-1)^{m-k}\big( (2a+b)m+am (2h-m)\big)+E_k(m,h)\\
\ge &a(3k-1)m^2-3ak^3-\big(ak+(2a+b)+2B_f\big)m+\big((2a+b)+2B_f\big)k,
\end{align*}
Here we have used the facts that $|2h-m|\le m$ and $|E_k(m,h)|\le 2(m-k)B_f$. For $m\le h\le 3m$, we have
$$f(h-m)=a(h-m)^2+b(h-m)+c+B(h-m),$$
hence
$$(2a+b)m+am (2h-m)-f(h-m)= 2am+b(2m-h)+a\big(2m^2-(2m-h)^2\big)-c-B(h-m),$$
which implies
\begin{align*}
F_k(n)=&(2a+b)k-3ak^3+ak \big(3m^2+(2h-m)\big)\\
&-(-1)^{m-k}\big( 2am+b(2m-h)+a\big(2m^2-(2m-h)^2\big)-c-B(h-m)\big)+E_k(m,h)\\
\ge &a(3k-2)m^2-3ak^3-\big(ak+(2a+b)+2B_f\big)m+\big((2a+b)+2B_f\big)k-|c|-B_f.
\end{align*}
Here we have used the facts that $|2m^2-(2m-h)^2|\le 2m^2$, $|2m-h|\le m$, $|B(h-m)|\le B_f$ and $|E_k(m,h)|\le 2(m-k)B_f$.

\medskip

Combining the above two inequalities for $F_k(n)$, with
\begin{align*}
 \widetilde{F}_k(m)=a(3k-2)m^2-3ak^3-\big(ak+(2a+b)+2B_f\big)m+\big((2a+b)+2B_f\big)k-|c|-B_f,
\end{align*}
we have $F_k(n)\ge \widetilde{F}_k(m)$ for all $0\le h\le 3m$. Since $m>k$, we can write $m=k+\ell$ with $\ell\ge 1$. Hence we have
\begin{align}\label{eqf0}
\frac{\widetilde{F}_k(k+\ell)}{a(3k-2)}&=\ell^2+\left(2k-1/3-\frac{b+2B_f+8a/3}{a(3k-2)}\right)\ell-(k+{2}/{3})-\frac{|c|+B_f+4a/3}{a(3k-2)}\nonumber\\
&=\ell^2+2\left(k-1/6-\frac{A_F}{3k-2}\right)\ell-(k+{2}/{3})-\frac{B_F}{3k-2}=:P_F^k(\ell),
\end{align}
by recalling that (see \eqref{eqabkf})
$$A_F=\frac{B+4a/3+b/2}{a}\;\;\text{and}\;\; B_F=\frac{B+4 a/3+|c|}{a}.$$
Solving the quadratic inequality $P_F^k(\ell)>0$ in variable $\ell\ge 1$, we have
\begin{align*}
\ell>\left({1}/{6}+\frac{A_F}{3k-2}-k\right)+\sqrt{\left({1}/{6}+\frac{A_F}{3k-2}-k\right)^2+\left(k+{2}/{3}+\frac{B_F}{3k-2}\right)}=\ell_F(k),
\end{align*}
(see \eqref{eqlfk}). In particular, if $\ell_F(k)<1$ then we have $F_k(n)>0$ for all $n\ge k(3k+1)/2+2k+1$. Note that $\ell_F(k)<1$ if and only if
$$\sqrt{\left(k-1/6-\frac{A_F}{3k-2}\right)^2+\left(k+{2}/{3}+\frac{B_F}{3k-2}\right)}<\left(k+5/6-\frac{A_F}{3k-2}\right).
$$
By suitable simplification and transformation, we find that if $ k(3k-2)>(2A_F+B_F)$, that is
$$k>\frac{1}{3}\left(1+\sqrt{1+3(2A_F+B_F)}\right)=k_F,$$
(see \eqref{eqabkf}), then $\ell_F(k)<1$. This completes the proof of Theorem \ref{main}.
\end{proof}
\subsection{The applications of the fundamental result}
In this subsection, we give the proofs of Theorems \ref{mth1} and \ref{mth2}.
\begin{proof}[Proof Theorems \ref{mth1} and \ref{mth2}]
In order to better use Proposition \ref{prom} and Theorem \ref{main} to derive a smaller lower bound for $\ell_F(k)$ and $k_F$, we define
$$C_{\alpha,\beta,\gamma}=\frac{\max_{n\ge 0}P_{\alpha,\beta,\gamma}(n)+\min_{n\ge 0}P_{\alpha,\beta,\gamma}(n)}{2},$$
and write
$$R_{\alpha,\beta,\gamma}(n)=\frac{n^2+(\alpha+\beta+\gamma)n}{2\alpha\beta\gamma}+C_{\alpha,\beta,\gamma}+B_{\alpha,\beta,\gamma}(n),$$
with $B_{\alpha,\beta,\gamma}(n)=P_{\alpha,\beta,\gamma}(n)-C_{\alpha,\beta,\gamma}$. Then, we have
$$B_{\alpha,\beta,\gamma}=\max_{n\ge 0}|B_{\alpha,\beta,\gamma}(n)|=\frac{\max_{n\ge 0}P_{\alpha,\beta,\gamma}(n)-\min_{n\ge 0}P_{\alpha,\beta,\gamma}(n)}{2}.$$
Denoting by $K_{\alpha,\beta,\gamma}=\lfloor k_F\rfloor$ and $L_{\alpha,\beta,\gamma}(k)=\lfloor \ell_F(k)\rfloor$.
To complete the proof of Theorems \ref{mth1} and \ref{mth2}, we created Table \ref{t1} for the above constants, with the assistance of {\bf Mathematica} (the code can be find in the Appendix).
\begin{table}[ht]
\resizebox{0.9\textwidth}{!}{
\begin{tabular}{c|c|c|c|c|c|c|c|c}
  \hline
   $(\alpha,\beta,\gamma)$~  & (1,2,3)& (1,2,5)& (1,2,7) & (1,3,4) & (1,3,5) & (1,4,9) & (2,3,5) & (2,3,7) \\
   \hline
  $C_{\alpha,\beta,\gamma}$~ &  17/24 &  27/40  & 41/56 &7/12   &  19/30  &  13/24  &  49/120 &  71/168 \\
  \hline
  $B_{\alpha,\beta,\gamma}$~ &  7/24  &  13/40  & 23/56 &5/12   &  11/30  &   7/12  &  71/120 &  97/168 \\
  \hline
    $K_{\alpha,\beta,\gamma}$~ &   3    &    4  &  5  &4    &    5    &     8   &     7   &     8   \\
 \hline
$L_{\alpha,\beta,\gamma}(1)$ &  15    &  23   & 35  & 29   &     33   &   99   &   82  &    110   \\
  \hline
$L_{\alpha,\beta,\gamma}(2)$ &   2  &   4   &  6 &5   &   6           &    22  &     18  &    25  \\
  \hline
$L_{\alpha,\beta,\gamma}(3)$ &   1  &   1   &  2 & 2   &   2           &   10  &   7  &    11  \\
  \hline
$L_{\alpha,\beta,\gamma}(4)$ &   0  &  1   &  1 & 1  &   1           &    4  &   3   &     5  \\
  \hline
$L_{\alpha,\beta,\gamma}(5)$ &   0  &   0   & 1  &0   &   1     &    2  &   1   &     3  \\
  \hline
$L_{\alpha,\beta,\gamma}(6)$ &   0  &    0  &  0 & 0  &   0   &    1  &   1  &     1 \\
  \hline
$L_{\alpha,\beta,\gamma}(7)$ &   0  &    0  &  0 & 0  &   0   &    1  &    1 &     1  \\
  \hline
$L_{\alpha,\beta,\gamma}(8)$ &   0  &    0  &  0  &0  &   0   &    1   &   0   &    1   \\
\hline
\end{tabular}
}
\captionof{table}{Data for the triples $(\alpha,\beta, \gamma)$ in Theorems \ref{mth1} and \ref{mth2}}\label{t1}
\end{table}

\medskip

Using Theorem \ref{main}, we obtain the following Proposition.
\begin{proposition}\label{promm}
Let $k$ and $n$ be positive integers, and let the triples $(\alpha,\beta, \gamma)$, $K_{\alpha,\beta,\gamma}$ and $L_{\alpha,\beta,\gamma}(k)$ be given as in Table \ref{t1}.
\begin{enumerate}
  \item For all $n<k(3k+1)/2+2k+1$, we have
${\rm TP}_{\alpha,\beta,\gamma}^{k}(n)=R_{\alpha,\beta,\gamma}(n-k(3k+1)/2)$.
  \item For all $n\ge k(3k+1)/2+2k+1$ with $k>  K_{\alpha,\beta,\gamma}$, we have ${\rm TP}_{\alpha,\beta,\gamma}^{k}(n)>0$.
  \item For all $n\ge k(3k+1)/2+(3\ell-1)(2k+\ell)/2$ with $1\le k\le  K_{\alpha,\beta,\gamma}$ and any $\ell> L_{\alpha,\beta,\gamma}(k)$, we have ${\rm TP}_{\alpha,\beta,\gamma}^{k}(n)>0$.
\end{enumerate}
\end{proposition}
Under Proposition \ref{promm}, to determine the positivity of ${\rm TP}_{\alpha,\beta,\gamma}^{k}(n)$, we need to verify the finite cases of $1\le k\le K_{\alpha,\beta,\gamma}$ with
$$k(3k+1)/2+2k+1\le n<k(3k+1)/2+(3 L_{\alpha,\beta,\gamma}(k)-1)(2k+L_{\alpha,\beta,\gamma}(k))/2,$$
and the cases of $k\ge 1$ with $k(3k+1)/2\le n<k(3k+1)/2+2k+1$. The first one can be carried out using {\bf Mathematica}, and for the latter one, it follows from the positivity of $R_{\alpha,\beta,\gamma}(n)$ since ${\rm TP}_{\alpha,\beta,\gamma}^{k}(n)=R_{\alpha,\beta,\gamma}(n-k(3k+1)/2)$, which can also be done using {\bf Mathematica}. This completes the proofs of Theorems \ref{mth1} and \ref{mth2}.
\end{proof}

\section{Final remarks}
We believe that our methods used to study the positivity of the tail of the pentagonal number series can also solve other conjectures involving theta series of Merca (see \cite{MR4309312, MR4469674,MR4458135} for examples).

Our approaches for Theorems \ref{mth1} and \ref{mth2} are elementary, and with the aid of {\bf Mathematica}. It would be very appealing to have a combinatorial interpretation of Theorems \ref{mth1} and \ref{mth2}. More importantly, it would be very interesting to find $q$-series identities like \eqref{AM-id} for Theorems \ref{th2} and \ref{th1}, for which one can provide direct proofs.

\medskip

In their study on averaged truncations of Jacobi's triple product identity,
Andrews and Merca~\cite[Question~(2)]{MR2946378} and Guo and
Zeng~\cite[Conjecture~6.1]{MR3007145} posed the following conjecture:
\begin{conjecture}\label{AMconj}
For positive integers $m,k,R,S$ with $1\le S<R/2$, the coefficient of $q^m$ in
\begin{equation*}
\frac{(-1)^{k-1}}{(q^S,q^{R-S},q^{R};q^{R})_{\infty}}
\sum_{0\le n<k}(-1)^nq^{\binom{n+1}{2}R-nS}(1-q^{(2n+1)S})
\end{equation*}
is non-negative.
\end{conjecture}
In 2015, Mao~\cite{MR3280682} and
Yee~\cite{MR3280681} independently proved Conjecture~\ref{AMconj}
by different means; Mao's proof is algebraic, while Yee's proof is
combinatorial. The conjecture was later also confirmed by
He, Ji, and Zang~\cite{MR3398853} using combinatorial arguments. More recently, Wang and Yee~\cite[Theorem~2.3]{MR3937794}, as well as Schlosser and the author~\cite[Theorem~1.1]{SZ2023}, confirmed Conjecture~\ref{AMconj} by establishing different $q$-series identities. More results on averaged truncations of theta series can be found in the
papers by Wang and Yee~\cite{MR4064775}, by Xia, Yee and Zhao~\cite{MR4338939},
by Xia~\cite{MR4388462} and by Yao~\cite{MR4497425}.

\medskip

In~\cite{MR4309312}, Merca proposes the following conjecture on truncations of Jacobi's triple product identity, which refines the above Conjecture~\ref{AMconj}.
\begin{conjecture}\label{conjt1}For $1\le S<R/2$ and $k\ge 1$, the theta series
\begin{equation}
\frac{(-1)^k}{(q^S,q^{R-S};q^R)_\infty}\sum_{j\ge k}(-1)^jq^{j(j+1)R/2-jS}(1-q^{(2j+1)S})
\end{equation}
has non-negative coefficients.
\end{conjecture}

Recall that a \emph{unimodal sequence} is a finite sequence
of real numbers that first increases and then decreases. In \cite{SZ2023}, Schlosser and the author posed the following conjecture.
\begin{conjecture}\label{conjsz}
For any $k\in\bn_0$, define $J_k^{\textsc{t}}(m,n)$ by
$$\sum_{n\ge 0}\sum_{m\in\bz}J_k^{\textsc{t}}(m,n)q^nz^m
=\frac{(-1)^k}{(z,q/z,q;q)_\infty}\sum_{0\le \ell\le k}(-1)^\ell
q^{\ell (\ell+1)/2}z^{-\ell}(1-z^{2\ell+1}),$$
Then, for any $k\in\bn_0$ and $n\in\bn$ the sequence
  $\big(J_k^{\textsc{t}}(m,n)\big)_{-n\le m\le n}$
%and $\big(J_k^{\textsc{g}}(m,n)\big)_{-n\le m\le n}$
is unimodal.
\end{conjecture}
Notice that
$$\sum_{n\ge 0}\sum_{m\in\bz}J_k^{\textsc{t}}(m,n)q^nz^m=(-1)^k-\frac{(-1)^k}{(z,q/z,q;q)_\infty}\sum_{\ell>k}(-1)^\ell
q^{\ell (\ell+1)/2}z^{-\ell}(1-z^{2\ell+1}),$$
by Jacobi's triple product identity. Motivated by Conjectures \ref{conjt1} and \ref{conjsz}, we propose the following conjecture involving two-variable theta series, which is supported by mathematical experiments.
\begin{conjecture}\label{conjt20}For any integers $k,d\ge 1$ and integer $n\ge 0$, we define that
$$\sum_{n\ge 0}\mathscr{L}_{k,n}^{d}(z)q^{n+\frac{k(k+1)}{2}}=\frac{(-1)^k}{(1-z)(qz,q/z;q)_d}\sum_{j\ge k}(-1)^{j}q^{\frac{j(j+1)}{2}}z^{-j}(1-z^{2j+1}).$$
Then the Laurent polynomials
$$\mathscr{L}_{k,n}^{d}(z)=\sum_{-n-k\le m\le n+k}J_{k,n}^d(m)z^m$$
has positive coefficients for any integers $k\ge 1$, $d\ge 2$ and $n\ge 0$. Furthermore, the sequences
  $\big(J_{k,n}^{d}(m)\big)_{-k-n\le m\le k+n}$ are unimodal for all integers $k>1$, $d>3$ and $n\ge 0$.
\end{conjecture}
\begin{remark}We have the following remarks.
\begin{enumerate}
  \item We note that $J_{k,n}^{d}(m)$
is symmetric in $m$, that is $J_{k,n}^{d}(m)=J_{k,n}^{d}(-m)$
for all integers $m,n\ge 0$, $d\ge 1$ and $k\ge 1$. Moreover, $J_{k,n}^d(n+k)=1$.
  \item With $q$ and $z$ replaced by $q^R$ and $q^S$, respectively, we observe that the factor $$(1-z)(qz,q/z;q)_2=(q^S;q^R)_3(q^{R-S};q^R)_2$$ appears in the infinite $(q^S,q^{R-S};q^R)_\infty$.
Therefore, Conjecture \ref{conjt20} refines Conjecture \ref{conjt1}.
\end{enumerate}

\end{remark}

\medskip

Finally, we expect that the methods described in this paper are also applicable to our Conjecture \ref{conjt20}, which we intend to address and solve in our future work.

\section*{Appendix}
\begin{lstlisting}[language=Mathematica,caption={{\bf Matematica} code for Table \ref{t1}}]
Clear["Global`*"];
a0=alpha;b0=beta;c0=gamma;
Ta=Table[SeriesCoefficient[1/((1-q^(a0))(1-q^(b0))(1-q^(c0))),{q,0,n}]-(a n^2+b n),{n,1,a0 b0 c0}];
c=(Max[Ta]+Min[Ta])/2; Bf=(Max[Ta]-Min[Ta])/2;
a=1/(2 a0 b0 c0); b=(a0+b0+c0) a;
A=(Bf+4 a/3+b/2)/a; B=(Bf+4a/3+Abs[c])/a;
kF=IntegerPart[(1+Sqrt[1+3(2A+B)])/3];
lF=Table[IntegerPart[(1/6+A/(3k-2)-k)+Sqrt[(1/6+A/(3k-2)-k)^2+(k+2/3+B/(3k-2))]],{k,1,kF}];
{c,Bf, kF}
lF
\end{lstlisting}

%\bibliographystyle{plain}
%\bibliography{testtqs}

\end{document}